\documentclass[12pt]{article}

\usepackage{latexsym}
\usepackage{amsfonts}
\usepackage{amscd}
\usepackage{amsmath, amsthm, amssymb}
\usepackage{bbm}
\usepackage[all]{xy}
\usepackage{cite}
\usepackage{sseq}
\usepackage{pgf}

\author{Thomas John Baird}
\title{ Moduli of flat SU(3)-bundles over a Klein bottle}

\newtheorem{thm}{Theorem}[section]
\newtheorem{cor}[thm]{Corollary}
\newtheorem{lem}[thm]{Lemma}
\newtheorem{prop}[thm]{Proposition}

\theoremstyle{definition}

\newtheorem{example}{Example}

\newcommand{\ignore}[1]{}

\newcommand{\id}{\mathbbmss{1}}
\newcommand{\Z}{\mathbb{Z}}
\newcommand{\R}{\mathbb{R}}
\newcommand{\C}{\mathbb{C}}

\newcommand{\Q}{\mathbb{Q}}

\newcommand{\gau}{\mathcal{G}}
\newcommand{\A}{\mathcal{A}}

\begin{document}


\maketitle

\begin{abstract}
In this short note, we compute the Betti numbers of the moduli stack of flat $SU(3)$-bundles over a Klein bottle. We also handle the general compact group case over $\R P^2$. In all cases the cohomology is found to be equivariantly formal, supporting a conjecture from the author's doctoral thesis. Our results also verify conjectural formulas obtained by Ho-Liu using Yang-Mills Morse theory.
\end{abstract}

\section{Introduction}

Let $M$ be a smooth, compact nonorientable 2-manifold. For a choice of Lie group $G$, one may study the set of group homomorphisms $Hom( \pi_1(M), G)$, topologized using the compact open topology. The group $G$ acts continuously on $Hom(\pi_1(M),G)$ by conjugation, and we deem two homomorphisms equivalent if they lie in the same orbit. The quotient stack $[Hom(\pi_1(M),G)/G]$ is well known to be isomorphic to the topological moduli stack of flat $G$-bundles over $M$ (see \cite{freed2007com} for example).

In previous work \cite{baird2008msf}, the author computed the stack cohomology $$H_{stack}^*([Hom(\pi_1(M),SU(2))/SU(2)];\Q) = H_{SU(2)}^*(Hom(\pi_1(M),SU(2));\Q)$$in the rank 2 case, where the conjugation action was found to be equivariantly formal (Recall that a $G$-space $X$ is called equivariantly formal if $H_G^*(X) \cong H^*(X) \otimes H^*(BG)$ as graded vector spaces). This motivated the conjecture that equivariant formality remains true when $SU(2)$ is replaced by any compact connected Lie group $G$.

In this paper we compute rational Betti numbers and establish the conjecture in the case that $G =SU(3)$ and $M$ is the Klein bottle. We also establish the conjecture for $M = \R P^2$ and arbitrary compact, connected $G$.



The case of $G=SU(3)$ has already been studied by Ho-Liu \cite{ho2008msf}, using Yang-Mills Morse theory. The Yang-Mills functional determines a Morse stratification of the space $\A$ of connections, whose semistable stratum deformation retracts onto the space $\A_{flat}$ of flat connections. Ho and Liu prove that the normal bundles of the Morse strata have vanishing equivariant Euler classes and introduce the term \emph{locally antiperfect} to describe this property. They conjecture that the stratification is also globally \emph{antiperfect} and produce under this assumption candidate Poincar\'e series for the moduli stacks. For a $G$-space $X$, the Poincar\'e series $P_t^G(X)$ is defined by  $$P_t^G(X) = \sum_{i=0}^{\infty} \dim(H^i_G(X;\Q)) t^i.$$ Let $\Sigma_n$ denoted the connected sum of $n+1$ copies of $\R P^2$. Ho-Liu propose that

\begin{equation}\label{confirm}
P_t^{SU(3)}( Hom(\pi_1( \Sigma_n), SU(3))) \stackrel{!}{=} \frac{P_t(SU(3))^n + (1+t^2+t^4)(t^3 + 2t^4 +t^5)^n}{(1-t^4)(1-t^6)}
\end{equation}
In this paper we confirm (\ref{confirm}) in the cases $n=0,1$. 

\emph{Notational convention}: Unless otherwise stated, cohomology is singular with $\Q$ coefficients.

\section{Equivariant formality}

We review the notion of equivariant formality.

Let $G$ be a compact, connected Lie group with universal $G$-bundle $G \rightarrow EG \rightarrow BG$ and let $X$ denote a compact, Hausdorff $G$-space. The equivariant cohomology of a $G$-space $X$ is defined to be the cohomology of the associated fibre bundle $X \rightarrow EG\times_G X \rightarrow BG$, $$H_G^*(X;\Q) = H^*(EG \times_G X;\Q).$$ The Serre spectral sequence $(E_*^{p,q})$ of the fibration $EG \times_G X$ satisfies $E_2^{p,q} \cong H^p(X) \otimes H^q(BG)$ and converges to $H_G^*(X)$. If the spectral sequence collapses at page 2, then $H_G^*(X) \cong H^*(X)\otimes H^*(BG)$ as graded vector spaces and we say the $G$-space $X$ is equivariantly formal. The cohomology ring $H^*(BG)$ is a polynomial ring with even degree generators, so $H^{odd}(BG) = 0$. Consequently, if $H^{odd}(X)=0$ then $X$ must be equivariantly formal.

Let $T \subset G$ denote a maximal torus in $G$. The $G$-space $X$ is $G$-equivariantly formal if and only if it is $T$-equivariantly formal under the restricted $T$-action. If $X^T$ denotes the $T$-fixed point set, by \cite{bor} IV 5.5 we have 
\begin{equation}\label{brosnan}
\dim H^*(X) \geq \dim H^*(X^T)
\end{equation}
with equality if and only if $X$ is equivariantly formal. 

\section{The projective plane}

The projective plane has fundamental group $\pi_1(\R P^2) \cong \Z/2\Z$. It follows that $ Hom(\pi_1(\R P^2), G) \cong \{ g \in G| g^2 = \id\}$ is exactly the set of square roots of the identity element. 

\begin{prop}
Let $G$ denote a compact, connected Lie group. The $G$-space $Hom(\pi_1(\R P^2), G)$ is equivariantly formal.
\end{prop}

\begin{proof}
The set of square roots of $\id \in G$ is clearly invariant under conjugation. Every conjugacy class intersects a maximal torus $T \subset G$ and the number of roots of $\id$ in $T$ is finite, equal to $2^{\text{rank}(T)}$. So $Hom(\pi_1(\R P^2), G)$ is a finite union of homogeneous spaces of the form $G/H$ where $\text{rank}(H) =\text{rank}(G)$. Such $G$-spaces are equivariantly formal because $H^{odd}(G/H) = 0$ (see \cite{greub1972cca} pg. 467).
\end{proof}

\begin{example}
The space $Hom(\pi_1(\R P^2), U(n))$ is isomorphic to $ \{ A \in U(n)| A^2 = \id\}$. Since $A^2 = \id$ implies that the eigenvalues of $A$ are $\pm 1$, by choosing (say) the +1 eigenspace, we obtain an isomorphism to a union of Grassmanians: $$Hom(\pi_1(\R P^2), U(n)) \cong \coprod_{k =0}^n Gr_k(\C^n) $$ similarly
$$Hom(\pi_1(\R P^2), SU(n)) \cong \coprod_{\text{k even}} Gr_k(\C^n) $$
In particular, $$P_t(Hom(\pi_1(\R P^2), SU(3))) = 1 +(1+t^2 + t^4). $$ confirming (\ref{confirm}) when $n=0$.
\end{example}

\section{The Klein bottle}

In this section we consider the moduli stack of rank 3 bundles over a Klein bottle, $Kl$. It will be useful to pass between structure groups $U(3)$ and $SU(3)$ using the following Lemma:

\begin{prop}\label{us}
We have an equation of Poincar\'e polynomials $$P_t( Hom_{id}(\pi_1(Kl),U(3))) = (1+t)P_t(Hom(\pi_1(Kl), SU(3))),$$ where $Hom_{id}(\pi_1(Kl),U(3))$ denotes the connected component of $Hom(\pi_1(Kl),U(3))$ containing the trivial homomorphism.
\end{prop}
We postpone the proof of \ref{us} until the end of the section.

The fundamental group of the Klein bottle $Kl$ has a presentation $\pi_1(Kl) \cong \{ a_0, a_1| a_0^2 = a_1^2\}$. It follows that $$Hom(\pi_1(Kl), G) \cong \{ (g_0,g_1) \in G^2| g_0^2=g_1^2\}.$$
For the rest of this section, set $G := SU(3)$ and $Z := Hom(\pi_1(Kl), G)$. The main result of this paper is
\begin{thm}\label{laugh}
The space $Hom(\pi_1(Kl), SU(3))$ has Poincar\'e polynomial $$ P_t(Hom(\pi_1(Kl), SU(3))) = (1+t^3)(1+t^5) + (1+t^2+t^4)(t^3+2t^4+t^5).$$ 
\end{thm}

\begin{cor}
The conjugation action of $SU(3)$ on $Hom(\pi_1(Kl), SU(3))$ is equivariantly formal. Consequently 
\begin{equation}\label{addedform}
P_t^{SU(3)}(Hom(\pi_1(Kl), SU(3))) = \frac{(1+t^3)(1+t^5) + (1+t^2+t^4)(t^3+2t^4+t^5)}{(1-t^4)(1-t^6)}
\end{equation}
in accordance with $(\ref{confirm})$. 	
\end{cor}

\begin{proof}
Let $T\subset G=SU(3)$ be a maximal torus. The $T$-fixed point set is $Hom(\pi_1(Kl), G)^T = Hom(\pi_1(Kl), T)$. Any such homomorphism must pass through the abelianization $Ab(\pi_1(Kl)) \cong H_1(Kl;\Z) \cong \Z \oplus \Z/2\Z$, so $$Hom(\pi_1(Kl), G)^T  \cong T \times T_2,$$ where $T_2 \cong \Z/2\Z \oplus \Z/2\Z$ is the 2-torsion subgroup of $T$. In particular $$\dim(Hom(\pi_1(Kl), G)^T) = 16= \dim(Hom(\pi_1(Kl), G))$$ which by (\ref{brosnan}) proves equivariant formality. Thus $P_t^{SU(3)}(Hom(\pi_1(Kl), SU(3))) = P_t(Hom(\pi_1(Kl), SU(3)))P_t(BSU(3))$, giving rise to (\ref{addedform}). 
\end{proof}

We will compute the Poincar\'e polynomial for $Z$ using the long exact sequence of a pair. Let $A \subset Z$ denote the subset of pairs $(k_0,k_1)$ such that $k_0$ and $k_1$ lie in some common maximal torus. The subspace $A$ is preserved by the $G$-action, and the stabilizers of points in $A$ all contain maximal tori of $G$. Applying (\cite{b}, Theorem 3.3) we produce a $G$-equivariant cohomology isomorphism 
\begin{equation}\label{realfirst}
\phi: G/T \times_{W} Z^T \rightarrow A.
\end{equation}
Here $T \subset G$ is a maximal torus, $G/T$ is a coset space, $Z^T$ is the set of $T$-fixed points, $W = N(T)/T$ is the Weyl group and $G/T \times_{W} Z^T$ is the orbit space of the product $G/T \times Z^T$ under the natural diagonal $W$ action. The map $\phi$ sends $\phi( gT, z) = g \cdot z$.

\begin{lem}
Let $ W = N(T)/T$ act on $T_2 := \{ t \in T| t^2 = \id\}$, $T$ and $G/T$ in the usual way. There is a cohomology isomorphism 
\begin{equation}\label{firstequ}
H^*(A;\Q) \cong H^*(G/T \times_W (T \times T_2);\Q)
\end{equation} 
with Poincar\'e polynomial 
\begin{equation}\label{oxfordtyep}
P_t( A) = (1+t^3)(1+t^5) + (1+t^2+t^4)(1+t)(1+t^3)
\end{equation}
\end{lem}   

\begin{proof}
The $T$-fixed point set $Z^T$ consists of pairs $(g_0,g_1) \in T^2$ such that $t_0^2 = t_1^2$ which is equivalent to $t_0t_1 \in T_2$, so $Z^T \cong T \times T_2$. Equation (\ref{firstequ}) then follows from (\ref{realfirst}). 

Choose the maximal torus to be the diagonal matrices in $SU(3)$. Then $T_2$ consists of the four matrices with diagonal entries $\{ (1,1,1), (-1,1,1), (1,-1,1), (1,1,-1)\}$. The Weyl group $W \cong S_3$ acts on $T_2$ by fixing the identity and permuting the remaining three element. It follows that $G/T \times_W (T \times T_2)$ has two components. The first component $G/T \times_W T$ has cohomology isomorphic to $SU(3)$ and contributes the first term of (\ref{oxfordtyep}). The second component is $G/T \times_{\Z_2}T$ where the $\Z_2$ action is generated by any of the order two elements in $W$. The Poincar\'e polynomials we need to know are $P_t(G/T) = (1+2t^2 +2t^4+t^6),$ $P_t( (G/T)/\Z_2) = P_t( \C P^2) = 1+t^2+t^4$, $P_t(T) = (1+t)^2$ and $P_t(T/\Z_2) = 1+t$. The result follows by a simple calculation.
\end{proof}

Equation (\ref{firstequ}) provides a description of $A$ that can be generalized to any compact Lie group $G$. The more complicated part is describing the complement $Z - A$.

\begin{lem}\label{xmasgb}
The Poincar\'e polynomial of the pair $(Z, A)$ satisfies 
\begin{equation}\label{inoct}
P_t(Z,A) \leq (1+t)(t^1+t^3+t^6+t^8),
\end{equation}
where the notation $p(t) \leq q(t)$ means that $q(t) - p(t)$ has no negative coefficients. 
\end{lem}

\begin{proof}
We will show that $Z-A$ is a $9$-dimensional orientable manifold, so by Poincar\'e duality $H^*(Z,A) \cong H^{(9-*)}(Z-A)$ and we can focus instead on $Z-A$. 

A pair $(g,h) \in Z-A$ is characterized by $g,h \in SU(3), gh \neq hg, g^2=h^2$. Since $g$ and $h$ both commute with $g^2$ but not with each other, we deduce that $g^2$ must have an eigenvalue $\lambda^2$ of multiplicity greater than one, and that $g$ and $h$ have eigenvalues $ \lambda, -\lambda, \lambda^{-2}$. It is easily verified that $g$, $h$ must share an eigenvector $\bar{x}$ such that $g \bar{x} = h \bar{x} = \lambda^{-2} \bar{x}$ and that up to scalar multiplication $\bar{x}$ is the only common eigenvector of $g$ and $h$.

Now let $P(\C^3)$ denote the set of Hermitian, rank 1 projection operators on $\C^3$. We define the flag variety $Fl_3 := \{ p=(p_1,p_2,p_3) \in P(\C^3)| p_i p_j = \delta_{i,j} p_i\}$ where $\delta_{i,j}$ denotes the Dirac delta. Define $\psi: Fl_3 \times U(1) \rightarrow SU(3)$ by $\psi(p, \lambda) = \lambda p_1 - \lambda p_2 + \lambda^2 p_3$.

Define $Y := \{ (p,q) \in Fl_3| p_3 = q_3 \text{ and } p_1q_1 \neq q_1p_1\}$. Then by the above discussion, the map $ \rho: Y \times U(1) \rightarrow Z-A$ defined by $\rho((p,q),\lambda) = (\psi(p,\lambda), \psi(q,\lambda))$ is a two sheeted cover. It follows that the induced map $H^*(Z-A) \hookrightarrow H^*( Y \times U(1))$ is injective. We can therefore obtain upperbounds of the Betti numbers of $Z-A$ by computing $H^*(Y)$.

Let $\pi: Y \rightarrow Fl_3$ denote the projection $\pi (p,q) = p$. An element $q$ in the fibre $\pi^{-1}(p)$ is uniquely determined by a choice $q_1 \in P(\C^3)$ with $q_1p_3 = 0$ and $q_1 \neq p_1,p_2$. Thus $\pi$ has fibres isomorphic to $ \C P^1 - 2 pts \cong \C^*$. In fact, if $L_1, L_2,L_3$ denote the tautological line bundles over $Fl_3$, it follows that $Y$ is isomorphic to $Hom(L_1, L_2)$ minus the zero section. We see that $Y \times U(1)$ is a $9$-dimensional orientable manifold. The Serre spectral sequence to the $\C^*$-bundle $Y \rightarrow Fl_3$ has $E_2$ page:
\\
\\
\sseqystep 1
\sseqentrysize= .8cm
\begin{sseq}{7}{2}
\ssdrop{1}
\ssmoveto 0 1
\ssdrop{1}
\ssmoveto 6 0
\ssdrop{1}
\ssmoveto 6 1
\ssdrop{1}
\ssmoveto 2 0
\ssdrop{2}
\ssmoveto 2 1
\ssdrop{2}
\ssmoveto 4 0
\ssdrop{2}
\ssmoveto 4 1
\ssdrop{2}
\ssmoveto 0 1
\ssarrow {2} {-1}
\ssmoveto 2 1
\ssarrow {2} {-1}
\ssmoveto 4 1
\ssarrow {2} {-1}
\end{sseq}

The boundary map from $E_2^{p,1} \cong H^p(Fl_3)$ to $E_2^{p+2,0} \cong H^{p+2}(Fl_3)$ is simply cup product by the first Chern class of $Hom(L_1, L_2)$. Here $H^*(Fl_3) \cong Q[x_1,x_2,x_3] / (\text{symmetric polynomials})$, $c_1(Hom(L_1,L_2)) = x_2-x_1$ and the $E_3$ page is:
\\
\\
\begin{sseq}{7}{2}
	\ssdrop{1}
	\ssmoveto 0 1
	\ssdrop{0}
	\ssmoveto 6 0
	\ssdrop{0}
	\ssmoveto 6 1
	\ssdrop{1}
	\ssmoveto 2 0
	\ssdrop{1}
	\ssmoveto 2 1
	\ssdrop{0}
	\ssmoveto 4 0
	\ssdrop{0}
	\ssmoveto 4 1
	\ssdrop{1}	
\end{sseq}

So $P_t(Y) = 1+t^2 +t^5+t^7$ so $P_t(Z-A) \leq (1+t^2+t^5+t^7)(1+t)$ by the Kunneth theorem, and Poincar\'e duality completes the proof. 
\end{proof}
In fact we will show in the course of proving Theorem \ref{laugh} that (\ref{inoct}) is an equality. 
We need one more Lemma for our calculation.

\begin{lem}\label{xmasday}
Modulo terms of order $4$ or higher, $$P_t(Z) = 1 + 2t^3 \text{  (mod } t^4).$$
\end{lem}

\begin{proof}
Let $\A$ denote the affine space of $U(3)$-connections on the trivial $\C^3$ bundle over $Kl$, $\A_{flat}$ the subspace of flat connections and let $\gau_b$ denote the group of gauge transformations based at one point. According to (\cite{baird2008msf} proof of Theorem 1.1) we have an injection $H_{\gau_b}^*(\A) \rightarrow H_{\gau_b}^*(\A_{flat})$ and an isomorphism $H_{\gau_b}^*(\A) \cong H^*(U(3))$, so

$$ P_t^{\gau_b}(\A_{flat}) = (1+t)(1+t^3)(1+t^5) + t^{-1}P_t^{\gau}(\A, \A_{flat}) $$

It was proven in \cite{hl2} and \cite{ho-liu-ramras} that the Yang-Mills functional determines a Morse stratification, $\A = \bigcup_{\mu = 0}^{\infty} \A_{\mu}$, where $\A_{\mu}$ is a connected submanifold with codimension $4\mu$ and orientable normal bundle and for which the open stratum $\A_0$ contains $\A_{flat}$ as a deformation retract. We deduce that $P_t^{\gau_b}(\A, \A_{flat}) = P_t^{\gau_b}(\A, \A_{0}) = t^4 \text{  (mod } t^5) $, so 

$$ P_t^{\gau_b}(\A_{flat}) = (1+t)(1+t^3)(1+t^5) + t^3 = 1+t+2t^3 \text{  (mod } t^4)$$

Finally, holonomy determines a homeomorphism $ \A_{flat}/ \gau_{0} \cong Hom_{id}(\pi_1(Kl), U(3))$, so the lemma follows by applying Lemma \ref{us}.
\end{proof}

\begin{proof}[Proof of Theorem \ref{laugh}]
We consider the long exact sequence of the pair $(Z,A)$ as the spectral sequence for the filtration, we get:
\\
\\
\sseqystep 1
\sseqxstep 1
\sseqentrysize= .8cm
\begin{sseq}{9}{2}
\ssmoveto 0 1
\ssdrop{*}
\ssmoveto 1 1
\ssdrop{*}
\ssmoveto 2 1
\ssdrop{*}
\ssmoveto 3 1
\ssdrop{*}
\ssmoveto 4 1
\ssdrop{0}
\ssmoveto 5 1
\ssdrop{*}
\ssmoveto 6 1
\ssdrop{*}
\ssmoveto 7 1
\ssdrop{*}
\ssmoveto 8 1
\ssdrop{*}
\ssmoveto 0 0
\ssdrop{2}
\ssmoveto 1 0 
\ssdrop{1}
\ssmoveto 2 0
\ssdrop{1}
\ssmoveto 3 0
\ssdrop{3}
\ssmoveto 4 0
\ssdrop{2}
\ssmoveto 5 0
\ssdrop{3}
\ssmoveto 6 0
\ssdrop{1}
\ssmoveto 7 0
\ssdrop{1}
\ssmoveto 8 0
\ssdrop{2}
\ssmoveto 9 0

\ssmoveto 0 0
\ssarrow {0} {1}
\ssmoveto 1 0
\ssarrow {0} {1}
\ssmoveto 2 0
\ssarrow {0} {1}
\ssmoveto 3 0
\ssarrow {0} {1}
\ssmoveto 7 0
\ssarrow {0} {1}
\ssmoveto 5 0
\ssarrow {0} {1}
\ssmoveto 6 0
\ssarrow {0} {1}
\ssmoveto 8 0
\ssarrow {0} {1}

\end{sseq}
\\
where each of the star entries is either 0 or 1 by Lemma \ref{xmasgb}. Lemma \ref{xmasday} implies that the first four stars are 1 and are killed by the boundary map. Inequality (\ref{brosnan}) then implies that the remaining four stars are also 1, but are not killed by the boundary map. Thus the $E_2$ page looks like
\\
\\
\sseqystep 1
\sseqentrysize= .8cm
\begin{sseq}{9}{2}
\ssmoveto 0 1
\ssdrop{0}
\ssmoveto 1 1
\ssdrop{0}
\ssmoveto 2 1
\ssdrop{0}
\ssmoveto 3 1
\ssdrop{0}
\ssmoveto 4 1
\ssdrop{0}
\ssmoveto 5 1
\ssdrop{1}
\ssmoveto 6 1
\ssdrop{1}
\ssmoveto 7 1
\ssdrop{1}
\ssmoveto 8 1
\ssdrop{1}
\ssmoveto 0 0
\ssdrop{1}
\ssmoveto 1 0 
\ssdrop{0}
\ssmoveto 2 0
\ssdrop{0}
\ssmoveto 3 0
\ssdrop{2}
\ssmoveto 4 0
\ssdrop{2}
\ssmoveto 5 0
\ssdrop{3}
\ssmoveto 6 0
\ssdrop{1}
\ssmoveto 7 0
\ssdrop{1}
\ssmoveto 8 0
\ssdrop{2}
\ssmoveto 9 0

\end{sseq}
\\
completing the proof.

\end{proof}

\begin{proof}[Proof of Proposition \ref{us}]
The identity component is identified with $$Hom_{id}(\pi_1(Kl), U(2)) \cong X := \{ (g_0,g_1) \in U(2)^2| g_0^2 = g_1^2, det(g_0)= det(g_1)\}.$$
It follows that the map $ \rho: X \rightarrow U(1)$, $\rho((g_0,g_1)) = det(g_0)$ is a fibre bundle with fibre $Z = Hom(\pi_1(Kl), SU(3))$. We deduce that $ \kappa: Z \times U(1) \rightarrow X$ sending $\kappa((g_0, g_1), \lambda) = (\lambda g_0. \lambda g_1)$ is a Galois cover. The deck transformation group $ \Z_3 = \{ \gamma \in U(1)| \gamma^3 = \id\}$ acts by $\gamma \cdot ((g_0, g_1), \lambda) = ((\gamma g_0, \gamma g_1), \gamma^2 \lambda)$. Thus $H^*(X) = H^*(Z \times U(1))^{\Z_3}$. 

To prove the proposition, we must show that $\Z_3$ acts trivially on the cohomology of $ Z \times U(1)$. The action on the $U(1)$ factor is by translation, which necessarily is cohomologically trivial. To understand the effect on $H^*(Z)$, notice that the action preserves the subspace $A \subset Z$ from (\ref{realfirst}) Thus the long exact sequence of the pair : $$ ...\rightarrow H^*(Z,A) \rightarrow H^*(Z) \rightarrow H^*(A) \rightarrow...$$ is acted on equivariantly by $\Z_3$. By Lemma \ref{xmasgb}, the cohomology groups $H^*(Z,A)$ have dimension 1 or 0 in all degrees, so $\Z_3$ must act trivially on $H^*(Z,A)$. The action on $A$ lifts to the model space $(G/T \times T \times T_2)^W$ where it acts by translation of the $T$ factor, which is an isotopy and hence trivial cohomologically. Since $\Z_3$ acts trivially on both $H^*(Z,A)$ and $H^*(A)$, we infer from the long exact sequence that it acts trivially also on $H^*(Z)$, completing the proof.
\end{proof}

\bibliographystyle{plain}

\bibliography{TomReferences}

\end{document}